\documentclass{amsart}
\title[de Rham Dolbeault Cohomology of  solvmanifolds]
{de Rham and Dolbeault Cohomology of   solvmanifolds with local systems}

\author{Hisashi Kasuya}

\usepackage{amssymb}
\usepackage{amsmath}
\usepackage{amscd}
\usepackage{amstext}
\usepackage{amsfonts}
\usepackage[all]{xy}

\theoremstyle{plain}

\theoremstyle{plain}

\theoremstyle{plain}

\theoremstyle{plain}
\newtheorem{theorem}{Theorem}[section] 
\theoremstyle{remark}
\newtheorem{remark}{Remark}
\theoremstyle{Main result}
\newtheorem{main result}{Main result}
\theoremstyle{lemma}
\newtheorem{lemma}[theorem]{Lemma}
\theoremstyle{definition}

\theoremstyle{proposition}
\newtheorem{proposition}[theorem]{Proposition}
\theoremstyle{corollary}
\newtheorem{corollary}[theorem]{Corollary}
\theoremstyle{remark}

\address[H.kasuya]{Graduate school of mathematical science university of tokyo japan }
\curraddr{}
\email{khsc@ms.u-tokyo.ac.jp}


\newcommand{\C}{\mathbb{C}}
\newcommand{\R}{\mathbb{R}}

\newcommand{\Z}{\mathbb{Z}}
\newcommand{\g}{\frak{g}}
\newcommand{\n}{\frak{n}}

\begin{document} 

\maketitle
\begin{abstract}
Let $G$ be a simply connected solvable Lie group with a lattice $\Gamma$ and the Lie algebra $\g$ and a representation $\rho:G\to GL(V_{\rho})$ whose  restriction on the nilradical is unipotent.
Consider the flat bundle $E_{\rho}$ given by $\rho$. By using "many" characters $\{ \alpha\}$ of $G$ and "many" flat line bundles $\{E_{\alpha}\}$ over $G/\Gamma$,
we show that an isomorphism 
\[\bigoplus_{\{ \alpha\}} H^{\ast}(\g, V_{\alpha}\otimes V_{\rho})\cong \bigoplus_{\{E_{\alpha}\}} H^{\ast}(G/\Gamma, E_{\alpha}\otimes E_{\rho})\]
 holds.
This isomorphism is a generalization of the well-known fact:"If $G$ is nilpotent and $\rho$ is unipotent then, the isomorphism $H^{\ast}(\g,  V_{\rho})\cong H^{\ast}(G/\Gamma,  E_{\rho})$ holds".
By this result, we construct an explicit finite dimensional cochain complex which compute the cohomology $H^{\ast}(G/\Gamma,  E_{\rho})$ of solvmanifolds even if the isomorphism $H^{\ast}(\g,  V_{\rho})\cong H^{\ast}(G/\Gamma,  E_{\rho})$ does not hold.
For Dolbeault cohomology of complex parallelizable solvmanifolds, we also prove an analogue of the above isomorphism result
which is a generalization of  computations of Dolbeault cohomology of complex parallelizable nilmanifolds.
By this isomorphism, we construct an explicit finite dimensional cochain complex which compute the Dolbeault cohomology of  complex parallelizable solvmanifolds.

\end{abstract}
\section{Background and Main results}
\subsection{Background}
We have nice theorem for de rham cohomology of nilmanifolds with local systems.

\begin{theorem}[due to \cite{Nom} or \cite{R}]\label{nill}
Let $N$ be a simply connected real nilpotent Lie group and $\frak n$ the Lie algebra of $N$.
Suppose $N$ has a lattice $\Gamma$.
Let $\rho:N\to GL( V_{\rho})$ be a finite dimensional unipotent representation.
We define the flat bundle $E_{\rho}=(N\times V_{\rho})/\Gamma$ given by the equivalent relation $(\gamma g, \rho(\gamma)v)\cong (g, v)$ for $g\in N$, $v\in V_{\rho}$, $\gamma\in \Gamma$.
Consider  the cochain complex $\bigwedge {\frak n}^{\ast}_{\C}\otimes V_{\rho}$ of Lie algebra (see \cite{OV} )
amd the canonical inclusion
\[\bigwedge \n^{\ast}_{\C}\otimes V_{\rho}\to A^{\ast}(N/\Gamma, E_{\rho}).
\]
Then this inclusion induces a cohomology isomorphism
\[H^{\ast}(\n, V_{\rho})\cong H^{\ast}(N/\Gamma, E_{\rho}).
\]
\end{theorem}
Some researchers tried to extend Theorem \ref{nill} for solvmanifolds.
In fact it is proved that for a simply connected solvable Lie group $G$ with the Lie algebra $\g$ admitting a lattice $\Gamma$ and a representation $\rho:G\to GL( V_{\rho})$, if:

(H) (\cite{Hatt}) The representation $\rho\oplus{\rm Ad}$ is triangular or,

(M) (\cite{Mosc}) The two images $(\rho\oplus{\rm Ad})(G)$ and $(\rho\oplus{\rm Ad})(\Gamma)$ have  same Zariski-closure in $GL(V_{\rho})\times {\rm Aut}(\g_{\C})$,\\
then the isomorphism $H^{\ast}(\g, V_{\rho})\cong H^{\ast}(G/\Gamma, E_{\rho})$ holds.
However in general  the isomorphism $H^{\ast}(\g, V_{\rho})\cong H^{\ast}(N/\Gamma, E_{\rho})$ does not hold.

As an Analogue of Theorem \ref{nill} we have the following theorem for Dolbeault cohomology of complex parallelizable nilmanifolds.

\begin{theorem}[due to \cite{Sak}]\label{dolni}
Let $N$ be a simply connected complex nilpotent  Lie group and $\n$ the Lie algebra (as a complex Lie algebra) of $N$.
Suppose $N$ has a lattice $\Gamma$.
Let $\sigma:N\to GL(V_{\sigma})$ be a finite dimensional holomorphic unipotent representation.
We also consider the anti-holomorphic representation $\bar\sigma: N\to GL(V_{\bar\sigma})$.
Define the flat holomorphic vector bundle $L_{\bar\sigma}=(N\times V_{\bar\sigma})/\Gamma$ over $G/\Gamma$ given by the equivalent relation $(\gamma g, \bar \sigma(\gamma)v)\cong (g, v)$ for $g\in N$, $v\in V_{\sigma}$, $\gamma\in \Gamma$.
We consider the Dobeault complex $(A^{\ast,\ast}(N/\Gamma,L_{\bar\sigma}), \bar\partial)$.
We regard  $\bigwedge \n^{\ast}\otimes V_{\sigma}$ as the subcomplex  of $(A^{0,\ast}(N/\Gamma,L_{\bar\sigma}), \bar\partial)$ which consists of the left-invariant "anti"-holomorphic forms with values in $L_{\bar\sigma}$.
Then the inclusion
\[\bigwedge \n^{\ast}\otimes V_{\sigma} \to A^{0,\ast}(N/\Gamma,L_{\bar\sigma})
\]
induces a cohomology isomorphism
\[H^{\ast}(\n, V_{\sigma})\cong H^{0,\ast}_{\bar\partial}(N, L_{\bar\sigma}).
\]
Hence since $N/\Gamma$ is complex parallelizable, we have an isomorphism
\[\bigwedge \C^{\dim N}\otimes H^{\ast}(\n, V_{\sigma})\cong H^{\ast,\ast}_{\bar\partial}(N, L_{\bar\sigma}).
\] 
\end{theorem}

It is desired that Theorem \ref{nill} and \ref{dolni} are generalized for solvmanifolds and we can compute the de Rham and Dolbeault cohomology of solvmanifolds even if the isomorphism  $H^{\ast}(\g, V_{\rho})\cong H^{\ast}(G/\Gamma, E_{\rho})$  (resp. $\bigwedge \C^{\dim N}\otimes H^{\ast}(\n, V_{\sigma})\cong H^{\ast,\ast}_{\bar\partial}(N, L_{\bar\sigma})$) does not holds.

\subsection{Main results}
The first purpose of this paper is to show new-type cohomology isomorphism theorems for solvmanifolds which are generalizations of Theorem \ref{nill} and \ref{dolni}.
These analogous each other.
We consider the "many" characters of $G$ and "many" line bundles over $G/\Gamma$.
In this paper we prove:
\begin{theorem}\label{drr}
Let $G$ be a simply connected  real solvable Lie group with a lattice $\Gamma$ and $\g$ the Lie algebra of $G$.
Let $N$ be the nilradical (i.e. maximal connected nilpotent normal subgroup)  of $G$.
Let ${\mathcal A}_{(G,N)}=\{\alpha\in {\rm Hom}(G,\C^{\ast}) \vert \alpha_{\vert_{N}}=1\}$
and  ${\mathcal A}_{(G,N)}(\Gamma)$ the set $\{E_{\alpha}\}$ of all the isomorphism classes of flat line bundles given by $\{V_{\alpha}\}_{\alpha\in {\mathcal A}_{(G,N)}}$.
Let $\rho:G\to GL(V_{\rho})$ be a representation.
For the nilradical $N$ of $G$, we assume that the restriction $\rho_{\vert_{N}}$ is a unipotent representation.
We consider the direct sum
\[\bigoplus_{\alpha\in {\mathcal A}_{(G,N)}}\bigwedge {\frak g}^{\ast}_{\C}\otimes V_{\alpha}\otimes V_{\rho}\]
of the Lie algebra cochain complexes.
We also consider the direct sum
\[\bigoplus_{E_{\alpha}\in {\mathcal A}_{(G,N)}(\Gamma)}A^{\ast}(G/\Gamma, E_{\alpha}\otimes E_{\rho}).
\]
Then the inclusion 
\[\bigoplus_{\alpha\in {\mathcal A}_{(G,N)}}\bigwedge {\frak g}^{\ast}_{\C}\otimes V_{\alpha}\otimes V_{\rho}\to \bigoplus_{E_{\alpha}\in {\mathcal A}_{(G,N)}(\Gamma)}A^{\ast}(G/\Gamma, E_{\alpha}\otimes E_{\rho}).\]
induces a cohomology isomorphism
\[\bigoplus_{\alpha\in {\mathcal A}_{(G,N)}}H^{\ast}(\g, V_{\alpha}\otimes V_{\rho})\cong H^{\ast}\cong \bigoplus_{E_{\alpha}\in {\mathcal A}_{(G,N)}(\Gamma)} H^{\ast}(G/\Gamma, E_{\alpha}\otimes E_{\rho}).
\]
\end{theorem}
We also prove:
\begin{theorem}\label{doll}
Let $G$ be a simply connected  complex solvable Lie group with a lattice $\Gamma$ and $\g$ the Lie algebra (as a complex Lie algebra) of $G$.
Let $N$ be the nilradical of $G$.
Let ${\mathcal B}_{(G,N)}=\{\alpha\in {\rm Hom}_{hol}(G,\C^{\ast}) \vert \alpha_{\vert_{N}}=1\}$
and  ${\mathcal B}_{(G,N)}(\Gamma)$ the set $\{L_{\bar\alpha}\}$ of all the isomorphism classes of holomorphic line bundles given by $\{V_{\bar\alpha}\}_{\alpha\in {\mathcal B}_{(G,N)}}$.
Let $\sigma:G\to GL(V_{\sigma})$ be a holomorphic  representation.
For the nilradical $N$ of $G$, we assume that the restriction $\sigma_{\vert_{N}}$ is a unipotent representation.
We consider the direct sum
\[\bigoplus_{\alpha\in{\mathcal B}_{(G,N)}} \bigwedge \g^{\ast}\otimes V_{\alpha}\otimes V_{\sigma}.
\]
of the Lie algebra cochain complexes.
We also consider the direct sum
\[\bigoplus_{L_{\bar\alpha}\in{\mathcal B}_{(G,N)}(\Gamma)} A^{0,\ast}(G/\Gamma, L_{\bar\alpha}\otimes L_{\bar\sigma})
\]
of Dolbeault complexes.

Then the inclusion
\[\bigoplus_{\alpha\in{\mathcal B}_{(G,N)}} \bigwedge \g^{\ast}\otimes V_{\alpha}\otimes V_{\sigma}\to \bigoplus_{L_{\bar\alpha}\in{\mathcal B}_{(G,N)}(\Gamma)} A^{0,\ast}(G/\Gamma, L_{\bar\alpha}\otimes L_{\bar\sigma})\]
induces a cohomology isomorphism
\[\bigoplus_{\alpha\in{\mathcal B}_{(G,N)}} H^{\ast}(\g, V_{\alpha}\otimes V_{\sigma})\cong \bigoplus_{L_{\bar\alpha}\in{\mathcal B}_{(G,N)}(\Gamma)}
H^{0,\ast}(G/\Gamma, L_{\bar\alpha}\otimes L_{\bar\sigma}).
\]
\end{theorem}

\begin{remark}
The correspondence ${\mathcal A}_{(G,N)}\to {\mathcal A}_{(G,N)}(\Gamma)$ (resp. ${\mathcal B}_{(G,N)}\to {\mathcal B}_{(G,N)}(\Gamma)$) is not $1$ to $1$.
This remark is very important for the case the isomorphism $H^{\ast}(\g, V_{\rho})\cong H^{\ast}(G/\Gamma, E_{\rho})$ (resp. $H^{\ast}(\g, V_{\sigma})\cong H^{0,\ast}(G/\Gamma, L_{\bar\sigma})$) does not hold.
\end{remark}

The second purpose of this paper is to construct a explicit finite dimensional cochain complex which compute the de Rham cohomology $H^{\ast}(G/\Gamma, E_{\rho})$ and the Dolbeault cohomology $H^{0,\ast}(G/\Gamma, L_{\bar\sigma})$ by using  Theorem \ref{drr}
and \ref{doll}.
We prove:
\begin{theorem}\label{finnn}
Let $G$ be a simply connected  real (resp complex) solvable Lie group and $\g$ the Lie algebra of $G$.
Define  ${\mathcal A}_{(G,N)}$ (resp. ${\mathcal B}_{(G,N)}$) as in Theorem \ref{drr} (resp Theorem  \ref{doll}).
Let $\rho:G\to GL(V_{\rho})$ (resp. $\sigma:G\to GL(V_{\sigma})$) be a representation with the assumption of Theorem \ref{drr} (resp Theorem  \ref{doll}).
We consider the direct sum
\[\bigoplus_{\alpha\in {\mathcal A}_{(G,N)}}\bigwedge {\frak g}^{\ast}_{\C}\otimes V_{\alpha}\otimes V_{\rho}\]
(resp. \[\bigoplus_{\alpha\in{\mathcal B}_{(G,N)}} \bigwedge \g^{\ast}\otimes V_{\alpha}\otimes V_{\sigma}\]
) of the Lie algebra cochain complexes.

Then there exists a finite dimensional subcomplex 
\[A^{\ast}\subset \bigoplus_{\alpha\in {\mathcal A}_{(G,N)}}\bigwedge {\frak g}^{\ast}_{\C}\otimes V_{\alpha}\otimes V_{\rho}\]
(resp. \[B^{\ast}\subset \bigoplus_{\alpha\in{\mathcal B}_{(G,N)}} \bigwedge \g^{\ast}\otimes V_{\alpha}\otimes V_{\sigma}\]
)
such that the inclusion induces a cohomology isomorphism.
\end{theorem}
By Theorem \ref{drr} (resp Theorem  \ref{doll}), we have the inclusion 
\[\iota: A^{\ast}\to \bigoplus_{E_{\alpha}\in {\mathcal A}_{(G,N)}(\Gamma)}A^{\ast}(G/\Gamma, E_{\alpha}\otimes E_{\rho})\]
(resp. 
\[\iota: B^{\ast}\to \bigoplus_{L_{\bar\alpha}\in{\mathcal A}_{(G,N)}(\Gamma)} A^{0,\ast}(G/\Gamma, L_{\bar\alpha}\otimes L_{\bar\sigma})\]
)
inducing a cohomology isomorphism.
Hence we have:
\begin{corollary}\label{Compu}
Let $A_{\Gamma}^{\ast}=\iota^{-1}(A^{\ast}(G/\Gamma,E_{\rho}))$ (resp. $B^{\ast}_{\Gamma}=\iota^{-1}(A^{0,\ast}(G/\Gamma,L_{\bar\sigma}))$.
Then we have an isomorphism
\[H^{\ast}(A^{\ast}_{\Gamma})\cong H^{\ast}(G/\Gamma, E_{\rho})\]
(resp.
\[H^{\ast}(B^{\ast})\cong H^{0,\ast}(G/\Gamma, L_{\bar\sigma})
\]
and hence $\bigwedge \C^{\dim G}\otimes H^{\ast}(B^{\ast})\cong H^{\ast,\ast}(G/\Gamma, L_{\bar\sigma})$).
\end{corollary}
Consider the adjoint representation ${\rm Ad}$.
Then the restriction  ${\rm Ad}_{\vert_{N}}$ is unipotent.
Hence by the above cochain complex $A^{\ast}_{\Gamma}$, we can compute the cohomology $H^{\ast}(G/\Gamma, E_{\rm Ad})$ on general  solvmanifolds.
The cohomology $H^{\ast}(G/\Gamma, E_{\rm Ad})\cong H^{\ast}(\Gamma, {\rm Ad})$ is important for studying the deformation of lattice $\Gamma$ in $G$.

\section{Preliminary: Jordan decompositions of representations}
Let $A\in GL_{n}(\C)$.
We denote by $A_{s}$ (resp. $A_{u}$) the semi-simple (resp. unipotent) part of $A$ for the Jordan decomposition (see \cite{hum} for the definition).
We will use the following facts.
\begin{lemma}\label{jse}
Let $N$ be  a simply connected nilpotent Lie group and $\varphi:N\to GL(V_{\varphi})$ a representation.
Then the map $\varphi^{\prime}:N\ni g\to (\varphi(g))_{s}$ is also a representation (see \cite{dek}).
Since $\varphi^{\prime}(N)$ is connected nilpotent group and consists of semi-simple elements,
the Zariski-closure of  $\varphi^{\prime}(N)$ is an algebraic torus (see \cite[Section 19]{hum})
 and hence $\varphi^{\prime}$ is diagonalizable.
\end{lemma}

\section{Proof of Theorem \ref{drr} }
\subsection{Cohomology of tori}
Let $A$ be a simply connected real abelian Lie group with a lattice $\Gamma$ and $\frak a$ the Lie algebra of $A$.
\begin{lemma}\label{aaag1}
Let $\rho:A\to GL (V_{\rho})$  be a  representation.
Suppose $\rho=\beta\otimes\phi$ such that $\beta$ is a character of  $A$ and $\phi$ is a unipotent representation.
Then we have:

If $\beta$ is non-trivial, then we have
\[H^{\ast}({\frak a},V_{\rho})=0.
\]

If the flat line bundle $E_{\beta}$ is non-trivial, then we have
\[H^{\ast}(A/\Gamma,E_{\rho})=0.\]
\end{lemma}
\begin{proof}
Suppose $\dim V_{\rho}=1$. Then if $\beta$ is non-trivial,
we can show $H^{\ast}({\frak a},V_{\rho})=H^{\ast}({\frak a},V_{\beta})=0$ by simple computation
and 
if $E_{\beta}$ is non-trivial, then we have
$H^{\ast}(A/\Gamma,E_{\rho})=H^{\ast}(\Gamma, \beta)=0$
by \cite[Lemma 2.1]{Kd}.

In case $\dim V_{\sigma}=n>1$,  by the triangulation of $\rho$, we have a $(n-1)$-dimensional $A$-submodule $V_{\rho^{\prime}}$ such that $V_{\rho}/V_{\rho^{\prime}}=V_{\beta}$.
Then by the  long exact sequence of cohomology of Lie algebra or group  (see \cite{OV}), the lemma follows inductively.  
\end{proof}

\begin{lemma}\label{ddecc}
Let $\rho:A\to GL(V_{\rho})$ be a  representation.
Then we have a basis of $V_{\rho}$  such that $\rho$ is represented by
\[\rho=\bigoplus_{i=1}^{k}\alpha_{i}\otimes\phi_{i}
\]
for characters   $\alpha_{i}$ of $G$ and unipotent representations $\phi_{i}$ of $G$.
\end{lemma}
\begin{proof}
For a character $\alpha$, we denote by $W_{\alpha}$ the subspace of $V_{\rho}$ consisting of the elements $w\in V_{\rho}$ such that for some positive integer $n$ we have $(\rho(a)-\alpha(a) I)^{n}w=0$ for any $a\in A$.
Since $A$ is abelian, we have a decomposition
\[V_{\rho}=W_{\alpha_{1}}\oplus\dots \oplus W_{\alpha_{k}}\]
by generalized eigenspace decomposition of $\rho(a)$ for all $a\in A$.
Let $\rho_{i}(a)=(\rho(a))_{\vert_{W_{\alpha_{i}}}}$.
Then we have $\rho=\rho_{1}\oplus \dots \oplus \rho_{k}$.
We have $(\rho_{i}(a))_{s}=\alpha_{i} I$.
Let $\phi_{i}(a)=(\rho_{i}(a))_{u}$.
By Lemma \ref{jse}, $\phi_{i}$ is a unipotent representation and we have $\rho_{i}(a)=(\rho_{i}(a))_{s}(\rho_{i}(a))_{u}=(\alpha_{i} \otimes \phi_{i})(a)$.
Hence the Lemma follows.

\end{proof}

Let $\{V_{\alpha}\}_{\alpha\in {\rm Hom}(A,\C^{\ast})}$ be the set of all $1$-dimensional representations of $A$ and  ${\mathcal H}(A/\Gamma)=\{E_{\beta}\}$ the set of all the isomorphism classes of 
flat line bundles given by $\{V_{\alpha}\}_{\alpha\in {\rm Hom}(A,\C^{\ast})}$.
We notice that the correspondence $\{V_{\alpha}\}_{\alpha\in {\rm Hom}(A,\C^{\ast})}\to {\mathcal H}(A/\Gamma)$ is not injective.
We consider the direct sums
\[\bigoplus_{\alpha\in {\rm Hom}(A,\C^{\ast})}\bigwedge {\frak a}^{\ast}_{\C}\otimes V_{\alpha}\otimes V_{\rho}\]
and 
\[\bigoplus_{E_{\alpha}\in {\mathcal H}(A/\Gamma)}A^{\ast}(A/\Gamma,E_{\alpha}\otimes E_{\rho}).\]
\begin{proposition}\label{abisR}
The  inclusion
\[\bigoplus_{\alpha\in {\rm Hom}(A,\C^{\ast})}\bigwedge {\frak a}^{\ast}_{\C}\otimes V_{\alpha}\otimes V_{\rho}\to \bigoplus_{E_{\alpha}\in {\mathcal H}(A/\Gamma)}A^{\ast}(A/\Gamma,E_{\alpha}\otimes E_{\rho})\]
induces a cohomology isomorphism.
\end{proposition}
\begin{proof}
Consider the  decomposition 
\[\sigma=\bigoplus_{i=1}^{k}\alpha_{i}\otimes\phi_{i}
\]
as the above lemma.
Then we have 
\[\bigoplus_{\alpha\in {\rm Hom}(A,\C^{\ast})}\bigwedge {\frak a}^{\ast}_{\C}\otimes V_{\alpha}\otimes V_{\rho}=\bigoplus_{\alpha\in {\rm Hom}(A,\C^{\ast})}\bigwedge {\frak a}^{\ast}\otimes  \bigoplus_{i=1}^{k} V_{\alpha\alpha_{i}}\otimes V_{\phi_{i}}\]
and 
\[\bigoplus_{E_{\alpha}\in {\mathcal H}(A/\Gamma)}A^{\ast}(A/\Gamma,E_{\alpha}\otimes E_{\rho})=\bigoplus_{E_{\alpha}\in {\mathcal H}(A/\Gamma)}A^{\ast}(A/\Gamma, \bigoplus_{i=1}^{k} E_{\alpha}\otimes E_{\alpha_{i}}\otimes E_{\phi_{i}}).
\]
By Theorem \ref{nill} and Lemma \ref{aaag1}, we have
\[H^{\ast}(\bigoplus_{E_{\alpha}\in {\mathcal H}(A/\Gamma)}A^{\ast}(A/\Gamma, E_{\alpha}\otimes E_{\rho}))\cong H^{\ast}(A/\Gamma, \bigoplus^{k}_{i=1} E_{\phi_{i}})\cong H^{\ast}({\frak a}, \bigoplus_{i=1}^{k} V_{\phi_{i}}). 
\]
By Lemma \ref{aaag1} we have
\[H^{\ast}(\bigoplus_{\alpha\in {\rm Hom}(A,\C^{\ast})}\bigwedge {\frak a}_{\C}^{\ast}\otimes V_{\alpha}\otimes V_{\rho})\cong H^{\ast}({\frak a}, \bigoplus_{i=1}^{k} V_{\phi_{i}}). 
\]
Hence the proposition follows.
\end{proof}
\subsection{Mostow bundle and spectral sequence}
Let $G$ be  a simply connected  solvable Lie group with a lattice $\Gamma$ and $\g$ be the Lie algebra of $G$.
Let $N$ be the nilradical of $G$.
It is known that $\Gamma\cap N$ is a lattice of $N$ and $\Gamma/\Gamma\cap N$ is a lattice of the abelian Lie group $G/N$ (see \cite{R}).
The solvmanifold $G/\Gamma$ is a fiber bundle
\[\xymatrix{
N/\Gamma\cap N=N\Gamma/\Gamma\ar[r]& G/\Gamma\ar[r]  &G/N\Gamma=(G/N)/(\Gamma/\Gamma\cap N)
}
\]
 over a torus with a nilmanifold $N/\Gamma\cap N$ as fiber.
We call this fiber bundle the Mostow bundle of $G/\Gamma$.
The structure group is $N\Gamma/\Gamma_{0}$ as left translations  where $\Gamma_{0}$ is the largest normal subgroup of $\Gamma$ which is normal in $N\Gamma$ (see \cite{St}).

Let $\rho:G\to GL(V_{\rho})$ be a representation such that the restriction $\rho_{\vert_{N}}$ is a unipotent representation.
For the Mostow bundle $p:G/\Gamma\to (G/N)/(\Gamma/\Gamma\cap N)$,
we define the vector bundle 
\[{\bf H}^{q}(N/\Gamma\cap N)=\sqcup_{x\in  (G/N)/(\Gamma/\Gamma\cap N)}H^{q}(p^{-1}(x),E_{\rho})\]
 over the torus $(G/N)/(\Gamma/\Gamma\cap N)$.
By Theorem \ref{nill}, 
 we have $H^{q}(p^{-1}(x), E_{\rho})\cong H^{q}(\n , V_{\rho})$.
Hence let $\Lambda_{q}:G/N\to GL(H^{q}(\n , V_{\rho}))$    be the representation induced by the extension $1\to N\to G\to G/N\to 1$, then
 we can regard ${\bf H}^{q}(N/\Gamma\cap N)$ as the flat  bundle $E_{\Lambda_{q}}$.
We consider the filtlation 
\[F^{p}\bigwedge ^{p+q}\g^{\ast}_{\C}=\{\omega\in \bigwedge^{p+q} \g^{\ast}_{\C}\vert \omega(X_{1},\dots, X_{p+1})=0 \,\, {\rm for}\,\, X_{1},\dots, X_{p+1}\in\n_{\C}\}.
\]
This filtration gives the filtration of the cochain complex $\bigwedge \g^{\ast}_{\C}\otimes V_{\rho}$ and the filtration of the de Rham complex $A^{\ast}(G/\Gamma,E_{\rho})$.
We consider the spectral sequence $E^{\ast,\ast}_{\ast}(\g)$ of $\bigwedge \g^{\ast}_{\C}\otimes V_{\rho}$ and the spectral sequence $E^{\ast,\ast}_{\ast}(G/\Gamma)$ of $A^{\ast}(G/\Gamma,E_{\rho})$.
Set $G/N=A$ and $\Gamma/\Gamma\cap N=\Delta$ and $\frak a=\g/\n$.
Then we have the commutative diagram
\[\xymatrix{
	E_{1}^{\ast,q}(\g)\ar[r]\ar[d]^{\cong}& E_{1}^{\ast,q}(G/\Gamma)\ar[d]^{\cong} \\
	\bigwedge {\frak a}_{\C}^{\ast}\otimes V_{\Lambda_{q}} \ar[r]&A^{\ast}(A/\Delta,E_{\Lambda_{q}})
 }
\]
(see \cite{Hatt}, \cite[Section 7]{R}).

\subsection{Proof of Theorem \ref{drr}}
\begin{proof}

Consider the spectral sequence $E_{\ast}^{\ast,\ast}(\g)$ of  
\[\bigoplus_{\alpha\in{\mathcal A}_{(G,N)}} \bigwedge \g^{\ast}_{\C}\otimes V_{\alpha}\otimes V_{\rho}\]
and the spectral sequence  $E_{\ast}^{\ast,\ast}(G/\Gamma)$ of
\[\bigoplus_{E_{\alpha}\in {\mathcal A}_{(G,N)}(\Gamma)}A^{\ast}(G/\Gamma, E_{\alpha}\otimes E_{\rho}).
\]
Set $A=G/N$ and $\Delta=\Gamma/\Gamma\cap N$ and $\frak a=\g/\n$.
Since we can identify ${\mathcal A}_{(G,N)}$ (resp. ${\mathcal A}_{(G,N)}(\Gamma) $) with ${\rm Hom}(A,\C^{\ast})$ (resp. ${\mathcal H}(A/\Delta)$),  we have the commutative diagram
\[\xymatrix{
E_{1}^{\ast,q}(\g)\ar[r]\ar[d]^{\cong}&E_{1}^{\ast,q}(G/\Gamma)\ar[d]^{\cong} \\
	\bigoplus_{\alpha\in {\rm Hom}(A,\C^{\ast})} \bigwedge {\frak a}_{\C}^{\ast}\otimes V_{\alpha}\otimes V_{\Lambda_{q}}  \ar[r]&\bigoplus_{E_{\alpha}\in{\mathcal H}(A/\Delta)} A^{0,\ast}(A/\Delta,E_{\alpha}\otimes E_{\Lambda_{q}}). 
 }
\]
By Proposition \ref{abisR}, the homomorphism $E_{1}^{\ast,\ast}(\g)\to E_{1}^{\ast,\ast}(G/\Gamma)$ induces a cohomology isomorphism and hence we have an isomorphism $E_{2}^{\ast,\ast}(\g)\cong E_{2}^{\ast,\ast}(G/\Gamma)$.
Hence the theorem follows.
\end{proof}

\section{Proof of Theorem  \ref{doll}}
\subsection{Dolbeault cohomology of tori}\label{To}
First we prove Theorem \ref{dolni}
by Sakane's Theorem \cite{Sak}.
\begin{proof}[Proof of Theorem \ref{dolni}]
In case $\dim V_{\sigma}=1$, $\sigma$ is trivial and the theorem follows from Sakane's Theorem \cite{Sak}.

In case $\dim V_{\sigma}=n>1$, since $\sigma$ is unipotent, we have a $(n-1)$-dimensional $G$-submodule $V_{\sigma^{\prime}}\subset  V_{\sigma}$ such that $V_{\sigma}/V_{\sigma^{\prime}}$ is the trivial submodule.
Then we have the spectral sequences
\[\xymatrix{
0\ar[r]& \bigwedge  \g^{\ast}\otimes V_{\sigma^{\prime}}\ar[r]&\bigwedge  \g^{\ast}\otimes V_{\sigma}\ar[r] &\bigwedge \g^{\ast}\otimes V_{\sigma}/V_{\sigma^{\prime}}\ar[r]&0
}
\]
and 
\[\xymatrix{
0\ar[r]&  A^{0,\ast}(G/\Gamma,L_{\bar\sigma^{\prime}})\ar[r]& A^{0,\ast}(G/\Gamma,L_{\bar\sigma})\ar[r] & A^{0,\ast}(G/\Gamma,L_{\bar\sigma}/L_{\bar\sigma^{\prime}})\ar[r]&0.
}
\]
We have the commutative diagram
\[\xymatrix{
0\ar[r]& \bigwedge  \g^{\ast}\otimes V_{\sigma^{\prime}}\ar[r]\ar[d]&\bigwedge  \g^{\ast}\otimes V_{\sigma}\ar[r]\ar[d] &\bigwedge  \g^{\ast}\otimes V_{\sigma}/V_{\sigma^{\prime}}\ar[r]\ar[d]&0\\
0\ar[r]&  A^{0,\ast}(G/\Gamma,L_{\bar\sigma^{\prime}})\ar[r]& A^{0,\ast}(G/\Gamma,L_{\bar\sigma})\ar[r] & A^{0,\ast}(G/\Gamma,L_{\bar\sigma}/L_{\bar\sigma^{\prime}})\ar[r]&0.
}
\]
Considering the long exact sequence of cohomologies, by the five lemma, the theorem follows inductively.
\end{proof}
Let $A$ be a simply connected complex abelian group with a lattice $\Gamma$ and $\frak a$ the Lie algebra of $A$.
\begin{lemma}\label{aaag}
Let $\sigma:A\to GL (V_{\sigma})$  be a holomorphic representation.
Suppose $\sigma=\beta\otimes\phi$ such that $\beta$ is a character of  $A$ and $\phi$ is a unipotent representation.
Then we have:

If $\beta$ is non-trivial, then we have
\[H^{\ast}({\frak a},V_{\sigma})=0.
\]

If the holomorphic line bundle $L_{\bar\beta}$ is non-trivial, then we have
\[H^{0,\ast}(A/\Gamma,L_{\bar\sigma})=0\]
\end{lemma}
\begin{proof}
In case $\dim V_{\sigma}=1$, the lemma is proved in \cite{Kd}.

In case $\dim V_{\sigma}=n>1$,  by the triangulation of $\sigma$, we have a $(n-1)$-dimensional $A$-submodule $V_{\sigma^{\prime}}$ such that $V_{\sigma}/V_{\sigma^{\prime}}=V_{\beta}$.
Then we have the exact sequence
\[\xymatrix{
0\ar[r]&  A^{0,\ast}(A/\Gamma,L_{\bar\sigma^{\prime}})\ar[r]& A^{0,\ast}(A/\Gamma,L_{\bar\sigma})\ar[r] & A^{0,\ast}(A/\Gamma,L_{\bar\sigma}/L_{\bar\sigma^{\prime}})\ar[r]&0
}.
\]
Considering the long exact sequence of cohomologies, the lemma follows inductively.
\end{proof}

By similar proof of Lemma \ref{ddecc}, we have the following lemma.
\begin{lemma}\label{DEC}
Let $\sigma:A\to GL(V_{\sigma})$ be a holomorphic representation.
Then we have a basis of $V_{\sigma}$  such that $\sigma$ is represented by
\[\sigma=\bigoplus_{i=1}^{k}\alpha_{i}\otimes\phi_{i}
\]
for  holomorphic characters $\alpha_{i}$  and   holomorphic unipotent representations $\phi_{i}$.
\end{lemma}

Let $\{V_{\alpha}\}_{\alpha\in {\rm Hom}_{hol}(A,\C^{\ast})}$ be the set of all $1$-dimensional holomorphic representations of $A$ and  ${\mathcal H}_{hol}(A/\Gamma)=\{L_{\bar\alpha}\}$ the set of all the isomorphism classes of 
holomorphic line bundles given by $\{V_{\alpha}\}_{\alpha\in {\rm Hom}_{hol}(A,\C^{\ast})}$.
We notice that the correspondence $\{V_{\alpha}\}_{\alpha\in {\rm Hom}_{hol}(A,\C^{\ast})}\to {\mathcal H}(A/\Gamma)$ is not injective.
We consider the direct sums
\[\bigoplus_{\alpha\in {\rm Hom}_{hol}(A,\C^{\ast})}\bigwedge {\frak a}^{\ast}\otimes V_{\alpha}\otimes V_{\sigma}\]
and 
\[\bigoplus_{L_{\bar\alpha}\in {\mathcal H}_{hol}(A/\Gamma)}A^{0,\ast}(A/\Gamma,L_{\bar\alpha}\otimes L_{\bar\sigma}).\]
\begin{proposition}\label{abis}
The  inclusion
\[\bigoplus_{\alpha\in {\rm Hom}_{hol}(A,\C^{\ast})}\bigwedge {\frak a}^{\ast}\otimes V_{\alpha}\otimes V_{\sigma}\to \bigoplus_{L_{\bar\alpha}\in {\mathcal H}_{hol}(A/\Gamma)}A^{0,\ast}(A/\Gamma,L_{\bar\alpha}\otimes L_{\bar\sigma})\]
induces a cohomology isomorphism.
\end{proposition}
\begin{proof}
By using Theorem \ref{dolni} and Lemma \ref{aaag} and \ref{DEC}, we can prove the proposition by similar argument of the proof of Proposition \ref{abisR}
\end{proof}
\subsection{Mostow bundle and spectral sequence}\label{Spe}
Let $G$ be  a simply connected complex solvable Lie group with a lattice $\Gamma$ and $\g$ be the Lie algebra of $G$.
Then
the mostow  bundle
\[\xymatrix{
N/\Gamma\cap N=N\Gamma/\Gamma\ar[r]& G/\Gamma\ar[r]  &G/N\Gamma=(G/N)/(\Gamma/\Gamma\cap N)
}
\]
is holomorphic.

Let $\sigma:G\to GL(V_{\sigma})$ be a representation such that the restriction $\sigma_{\vert_{N}}$ is a unipotent representation.
For the Mostow bundle $p:G/\Gamma\to (G/N)/(\Gamma/\Gamma\cap N)$,
we define the vector bundle 
\[{\bf H}^{0,q}(N/\Gamma\cap N)=\sqcup_{x\in  (G/N)/(\Gamma/\Gamma\cap N)}H^{0,q}(p^{-1}(x),L_{\bar\sigma})\]
 over the torus $(G/N)/(\Gamma/\Gamma\cap N)$.
By Theorem \ref{dolni}, 
 we have $H^{0,q}(p^{-1}(x), L_{\bar\sigma})\cong H^{0,q}(\n , V_{\sigma})$.
Hence let $\Lambda_{q}:G/N\to GL(H^{q}(\n , V_{\sigma}))$    be the representation induced by the extension $1\to N\to G\to G/N\to 1$, then
 we can regard ${\bf H}^{0,q}(N/\Gamma\cap N)$ as the flat holomorphic bundle $L_{\bar\Lambda}$.
We consider the filtration 
\[F^{p}\bigwedge ^{p+q}\g^{\ast}=\{\omega\in \bigwedge^{p+q} \g^{\ast}\vert \omega(X_{1},\dots, X_{p+1})=0 \,\, {\rm for}\,\, X_{1},\dots, X_{p+1}\in\n\}.
\]
This filtration gives the filtration of the cochain complex $\bigwedge \g^{\ast}\otimes V_{\sigma}$ and the filtration of the Dolbeault complex $A^{0,\ast}(G/\Gamma,L_{\bar\sigma})=C^{\infty}(G/\Gamma,L_{\bar\sigma})\otimes \bigwedge \g^{\ast}$.
We consider the spectral sequence $\,_{Dol}E^{\ast,\ast}_{\ast}(\g)$ of $\bigwedge \g^{\ast}\otimes V_{\sigma}$ and the spectral sequence $\,_{Dol}E^{\ast,\ast}_{\ast}(G/\Gamma)$ of $A^{0,\ast}(G/\Gamma,L_{\bar\sigma})$.
Set $G/N=A$ and $\Gamma/\Gamma\cap N=\Delta$ and $\frak a=\g/\n$.
By Borel's result \cite[Appendix 2]{Hir}, we have the commutative diagram
\[\xymatrix{
	\,_{Dol}E_{1}^{\ast,q}(\g)\ar[r]\ar[d]^{\cong}& \,_{Dol}E_{1}^{\ast,q}(G/\Gamma)\ar[d]^{\cong} \\
	\bigwedge {\frak a}^{\ast}\otimes V_{\Lambda_{q}} \ar[r]&A^{0,\ast}(A/\Delta,L_{\bar\Lambda_{q}}).
 }
\]
\subsection{Proof of theorem}
\begin{proof}

Consider the spectral sequence $\,_{Dol}E_{\ast}^{\ast,\ast}(\g)$ of  
\[\bigoplus_{\alpha\in{\mathcal B}_{(G,N)}} \bigwedge \g^{\ast}\otimes V_{\alpha}\otimes V_{\sigma}\]
and the spectral sequence  $\,_{Dol}E_{\ast}^{\ast,\ast}(G/\Gamma)$ of
\[\bigoplus_{L_{\bar\alpha}\in{\mathcal B}_{(G,N)}(\Gamma)} A^{0,\ast}(G/\Gamma, L_{\bar\alpha}\otimes L_{\bar\sigma}).
\]
Set $A=G/N$ and $\Delta=\Gamma/\Gamma\cap N$ and $\frak a=\g/\n$.
Since we can identify ${\mathcal B}_{(G,N)}$ (resp. ${\mathcal B}_{(G,N)}(\Gamma) $) with ${\rm Hom}_{hol}(A,\C^{\ast})$ (resp. ${\mathcal H}_{hol}(A/\Delta)$ as Section \ref{To}).   we have the commutative diagram
\[\xymatrix{
	\,_{Dol}E_{1}^{\ast,\ast}(\g)\ar[r]\ar[d]^{\cong}& \,_{Dol}E_{1}^{\ast,\ast}(G/\Gamma)\ar[d]^{\cong} \\
	\bigoplus_{\alpha\in {\rm Hom}_{hol}(A,\C^{\ast})} \bigwedge {\frak a}^{\ast}\otimes V_{\alpha}\otimes V_{\Lambda}  \ar[r]&\bigoplus_{L_{\alpha}\in{\mathcal H}_{hol}(A/\Delta)} A^{0,\ast}(A/\Delta,L_{\bar\alpha}\otimes L_{\bar\Lambda}). 
 }
\]
By Proposition \ref{abis}, the homomorphism $\,_{Dol}E_{1}^{\ast,\ast}(\g)\to \,_{Dol}E_{1}^{\ast,\ast}(G/\Gamma)$ induces a cohomology isomorphism and hence we have an isomorphism $\,_{Dol}E_{2}^{\ast,\ast}(\g)\cong \,_{Dol}E_{2}^{\ast,\ast}(G/\Gamma)$.
Hence the theorem follows.
\end{proof}

\section{Construction of finite cochain complex (de Rham case)}\label{cos}

We will use the following proposition.
\begin{proposition}{\rm (\cite[Proposition 3.3]{dek})}\label{Cde}
Let $G$ be a simply connected  solvable Lie group $G$ and $N$  the nilradical  of $G$.
Then we have a simply connected nilpotent subgroup $C\subset G$ such that $G=C\cdot N$.
\end{proposition}

\begin{remark}\label{CCO}
This proposition is given by the decomposition (not necessarily direct sum) $\g={\frak c}+\n$ (see \cite[Theorem 2.2]{Dek}).
Since this decomposition is compatible with any field (see \cite[Theorem 2.2]{Dek}), if $G$ is complex Lie group we can take a subgroup $C$ also complex.
\end{remark}

Let $G$ be a simply connected solvable Lie group and $\g$ be the Lie algebra of $G$.
Let $N$ be the nilradical  of $G$.
Let $\rho:G\to GL(V_{\rho})$ be a representation.
Suppose the restriction $\rho_{\vert_{N}}$ is unipotent.
We consider the direct sum
\[\bigoplus_{\alpha\in {\mathcal A}_{(G,N)}}\bigwedge {\frak g}^{\ast}_{\C}\otimes V_{\alpha}\otimes V_{\rho}.\]
Then we have the $G$-action on this cochain complex via $\bigoplus {\rm Ad}\otimes \alpha\otimes \rho$.
Since this action is extension of the Lie derivation,
the induced action on the cohomology is trivial.
Consider  the semi-simple part 
\[((\bigoplus_{\alpha\in{\mathcal A}_{(G,N)}} {\rm Ad}\otimes \alpha\otimes \rho)(g))_{s}=\bigoplus_{\alpha\in{\mathcal A}_{(G,N)}} ({\rm Ad}_{g})_{s}\otimes \alpha(g)\otimes (\rho(g))_{s}.
\]

Take a simply connected nilpotent subgroup $C\subset G$ as Proposition \ref{Cde}.
 Since $C$ is nilpotent, the map 
\[\Phi:C\ni c \mapsto \bigoplus_{\alpha\in{\mathcal A}_{(G,N)}} ({\rm Ad}_{g})_{s}\otimes \alpha(g)\otimes (\rho(g))_{s} \in {\rm Aut}(\bigoplus_{\alpha\in{\mathcal A}_{(G,N)}}\bigwedge {\frak g}^{\ast}_{\C}\otimes V_{\alpha}\otimes V_{\rho})\]
is a homomorphism.
We denote by 
\[(\bigoplus_{\alpha\in{\mathcal A}_{(G,N)}}\bigwedge {\frak g}^{\ast}_{\C}\otimes V_{\alpha}\otimes V_{\rho})^{\Phi(C)}
\]
the subcomplex consisting of the $\Phi(C)$-invariant elements.
\begin{lemma}\label{incis}
The inclusion
\[(\bigoplus_{\alpha\in{\mathcal A}_{(G,N)}}\bigwedge {\frak g}^{\ast}_{\C}\otimes V_{\alpha}\otimes V_{\rho})^{\Phi(C)}\subset \bigoplus_{\alpha\in{\mathcal A}_{(G,N)}}\bigwedge {\frak g}^{\ast}_{\C}\otimes V_{\alpha}\otimes V_{\rho}
\]
induces a cohomology isomorphism.
\end{lemma}

\begin{proof}
Since the induced $G$-action on the cohomology $H^{\ast}(\bigoplus_{\alpha\in{\mathcal A}_{(G,N)}}\bigwedge {\frak g}^{\ast}_{\C}\otimes V_{\alpha}\otimes V_{\rho})$ is trivial  and  $\Phi(C)$-action is semi-simple part of $G$-action, the induced $\Phi(C)$-action on the  cohomology $H^{\ast}(\bigoplus_{\alpha\in{\mathcal A}_{(G,N)}}\bigwedge {\frak g}^{\ast}_{\C}\otimes V_{\alpha}\otimes V_{\rho})$ is also trivial and hence
\[H^{\ast}(\bigoplus_{\alpha\in{\mathcal A}_{(G,N)}}\bigwedge {\frak g}^{\ast}_{\C}\otimes V_{\alpha}\otimes V_{\rho})^{\Phi(C)}=H^{\ast}(\bigoplus_{\alpha\in{\mathcal A}_{(G,N)}}\bigwedge {\frak g}^{\ast}_{\C}\otimes V_{\alpha}\otimes V_{\rho}).
\]
Since $\Phi$ is diagonalizable, we have
\[H^{\ast}((\bigoplus_{\alpha\in{\mathcal A}_{(G,N)}}\bigwedge {\frak g}^{\ast}_{\C}\otimes V_{\alpha}\otimes V_{\rho})^{\Phi(C)})=H^{\ast}(\bigoplus_{\alpha\in{\mathcal A}_{(G,N)}}\bigwedge {\frak g}^{\ast}_{\C}\otimes V_{\alpha}\otimes V_{\rho})^{\Phi(C)}.
\]
Hence the lemma follows.
\end{proof}

The subcomplex $(\bigoplus_{\alpha\in{\mathcal A}_{(G,N)}}\bigwedge {\frak g}^{\ast}_{\C}\otimes V_{\alpha}\otimes V_{\rho})^{\Phi(C)}$ is desired subcomplex $A^{\ast}$ as in Theorem \ref{finnn}.
By using certain basis, we see that this complex is finite dimensional and write down the subcomplex $A^{\ast}_{\Gamma}$ as Corollary \ref{Compu} explicitly.

We have a basis $X_{1},\dots,X_{n}$ of $\g_{\C}$ such that $({\rm Ad}_{c})_{s}={\rm diag} (\alpha_{1}(c),\dots,\alpha_{n}(c))$ for $c\in C$.
Let $x_{1},\dots, x_{n}$ be the basis of $\g_{\C}^{\ast}$ which is dual to $X_{1},\dots ,X_{n}$.
We have a basis $v_{1},\dots ,v_{m}$ of $V_{\rho}$ such that 
$(\rho(c))_{s}={\rm diag}(\alpha_{1}^{\prime}(c),\dots ,\alpha^{\prime}_{m}(c))$ for any $c\in C$.
Let $v_{\alpha}$ be a basis of $V_{\alpha}$ for each character $\alpha \in{\mathcal A}_{(G,N)}$.
By $G=C\cdot N$, we have $G/N=C/C\cap N$ and hence we have ${\mathcal A}_{(G,N)}={\mathcal A}_{C,C\cap N}=\{\alpha\in {\rm Hom}(C,\C^{\ast}) \vert \alpha_{\vert_{C\cap N}}=1\}$.

For a multi-index $I=\{i_{1},\dots ,i_{p}\}$ we write $x_{I}=x_{i_{1}}\wedge\dots \wedge x_{i_{p}}$,  and $\alpha_{I}=\alpha_{i_{1}}\cdots \alpha_{i_{p}}$.
We consider the basis 
\[\{ x_{I} \otimes v_{\alpha}\otimes v_{k}\}_{I\subset \{1,\dots, n\}, \alpha\in {\mathcal A}_{C,C\cap N}, k\in \{1,\dots, m \}}
\]
of $\bigoplus_{\alpha\in {\mathcal A}_{C,C\cap N}}\bigwedge {\frak g}^{\ast}_{\C}\otimes V_{\alpha}\otimes V_{\rho}$.
Since the action
\[\Phi:C\to{\rm Aut}(\bigoplus_{\alpha}\bigwedge {\frak g}^{\ast}\otimes V_{\alpha}\otimes V_{\rho})\]
is  the semi-simple part of $(\bigoplus {\rm Ad}\otimes \alpha\otimes \rho)_{\vert_C}$, we have
\[
\Phi(a)(x_{I} \otimes v_{\alpha}\otimes v_{k})
=\alpha^{-1}_{I}\alpha \alpha_{k}^{\prime} x_{I}\otimes v_{\alpha}\otimes v_{k}.
\]

Hence we have
\begin{multline*}
((\bigoplus_{\alpha}\bigwedge {\frak g}^{\ast}_{\C}\otimes V_{\alpha}\otimes V_{\rho})^{\Phi(C)}\\
=\langle x_{I}\otimes v_{\alpha_{I}\alpha^{\prime-1}_{k}} \otimes v_{k}\otimes \rangle_{I\subset \{1,\dots, n\},  k\in \{1,\dots, m \}}\\
=\bigwedge \langle x_{1}\otimes v_{\alpha_{1}},\dots, x_{n} \otimes v_{\alpha_{n}} \rangle \otimes \langle v_{\alpha^{\prime -1}_{1}}\otimes v_{1}, \dots ,v_{\alpha^{\prime -1}_{m}}\otimes v_{m} \rangle.
\end{multline*}
Finally we construct a finite dimensional complex $A_{\Gamma}^{\ast}$ which computes the de Rham cohomology $H^{\ast}(G/\Gamma,  E_{\rho})$.

\begin{corollary}
Let $A^{\ast}_{\Gamma}$ be the subcomplex of $(\bigoplus_{\alpha}\bigwedge {\frak g}_{\C}^{\ast}\otimes V_{\alpha}\otimes V_{\rho})^{\Phi(C)}$ defined as
\[A^{\ast}_{\Gamma}=\langle x_{I}\otimes v_{\alpha_{I}\alpha^{\prime-1}_{k}} \otimes v_{k}\vert (\alpha_{I}\alpha^{\prime-1}_{k} )_{\vert_{\Gamma}}=1\rangle.
\]
Then we have an isomorphism 
\[H^{\ast}(A^{\ast}_{\Gamma})\cong H^{\ast}(G/\Gamma,  E_{\rho}).
\]
\end{corollary}
\begin{proof}
Consider the inclusion 
\[\iota:(\bigoplus_{\alpha\in {\mathcal A}_{G, N}}\bigwedge {\frak g}^{\ast}_{\C}\otimes V_{\alpha}\otimes V_{\rho})^{\Phi(C)}\to  \bigoplus_{E_{\alpha}\in {\mathcal A}_{(G,N)}(\Gamma)}A^{\ast}(G/\Gamma, E_{\alpha}\otimes E_{\rho}).\]
$\iota(x_{I}\otimes v_{\alpha_{I}\alpha^{\prime-1}_{k}} \otimes v_{k})\in A^{\ast}(G/\Gamma,  E_{\rho})$ if and only if $(\alpha_{I}\alpha^{\prime-1}_{k} \rho)_{\vert_{\Gamma}}=\rho_{\vert_{\Gamma}}$.
Hence we have $\iota^{-1}(A^{\ast}(G/\Gamma,  E_{\rho}))=A^{\ast}_{\Gamma}$.
\end{proof}

\begin{corollary}
We consider the following conditions:

$(\Diamond_{1})$ For each multi-index $I=\{i_{1},\dots ,i_{p}\}$ and $k\in \{1,\dots ,m\}$, the character $\alpha_{I}\alpha^{\prime-1}_{k}$ is trivial if and only if the restriction $({\alpha_{I}\alpha^{\prime-1}_{k}})_{\vert_{\Gamma}}$ is trivial.

$(\Diamond_{2})$For each multi-index $I=\{i_{1},\dots ,i_{p}\}$ and $k\in \{1,\dots ,m\}$, the character $\alpha_{I}\alpha^{\prime-1}_{k}$ is trivial or non-unitary.

If the condition $(\Diamond_{1})$ or $(\Diamond_{2})$ holds,  then we have an isomorphism
\[H^{\ast}(\g, V_{\rho})\cong H^{\ast}(G/\Gamma,  E_{\rho}).
\]

\end{corollary}
\begin{proof}
If the condition $(\Diamond_{1})$ holds, then we have $A^{\ast}_{\Gamma}=(\bigwedge\g^{\ast}_{\C}\otimes V_{\rho})^{\Phi(C)}$.
Hence we have
\[H^{\ast}(G/\Gamma,E_{\rho})\cong H(\bigwedge\g^{\ast}_{\C}\otimes V_{\rho})^{\Phi(C)} \cong H^{\ast}(\g, V_{\rho}).\]
The condition $(\Diamond_{2})$ is special case of the condition $(\Diamond_{1})$.
Hence the corollary follows.

\end{proof}
\begin{remark}
For a representation $\rho:G\to GL(V_{\rho})$ such that the restriction $\rho_{\vert_{N}}$ is trivial,
the condition (M) (resp. (H)) in Section 1 is a special case of the condition $(\Diamond_{1})$ (resp $(\Diamond_{2})$)
\end{remark}
\begin{remark}\label{shd}
Let $\frak c$ be the Lie algebra of $C$.
Take a subvector $V\subset {\frak c}$ (not necessarily Lie algebra) such that ${\frak g}=V\oplus\n$.
Then we define the map
\[{\rm ad}_{s}:\g=V\oplus\n\ni A+X\mapsto ({\rm ad}_{A})_{s}\in D(\g)\]
where $({\rm ad}_{A})_{s}$ is the semi-simple part of ${\rm ad}_{A}$ and $D(\g)$ is the Lie algebra of derivations of $\g$.
This map is a Lie algebra homomorphism and a diagonalizable representation  (see \cite{DER} and \cite{K2}).
Let ${\rm Ad}_{s}:G\to {\rm Aut}(\g)$ be the extension of ${\rm ad}_{s}$.
Then this map is identified with the map
\[G=C\cdot N\ni c\cdot n\mapsto ({\rm Ad}_{c}) \in{\rm Aut}(\g).
\]
We define the Lie algebra ${\frak u}_{G}\subset D(\g)\ltimes \g$ as
\[{\frak u}_{G}=\{X-{\rm ad}_{sX}\vert X\in \g\}.\] 
Consider the above basis $\{x_{1},\dots,x_{n}\}$ of $\g^{\ast}_{\C}$.
Then in \cite{K2} the author showed that we have an isomorphism
\[\bigwedge \langle x_{1}\otimes v_{\alpha_{1}},\dots, x_{n}\otimes v_{\alpha_{n}} \rangle\cong \bigwedge({\frak u}_{G}\otimes \C)^{\ast}.\]
(This fact gives the new developments of de Rham homotopy theory on solvmanifolds. See \cite{K2}.)
Hence we can regard 
\[(\bigoplus_{\alpha\in {\mathcal A}_{G, N}}\bigwedge {\frak g}^{\ast}_{\C}\otimes V_{\alpha}\otimes V_{\rho})^{\Phi(C)}=\bigwedge \langle x_{1}\otimes v_{\alpha_{1}},\dots, x_{n}\otimes v_{\alpha_{n}} \rangle \otimes \langle v_{\alpha^{\prime -1}_{1}}\otimes v_{1}, \dots ,v_{\alpha^{\prime -1}_{m}}\otimes v_{m} \rangle\] as the cochain complex of nilpotent  Lie algebra of ${\frak u}_{G}$ with values in some  representation.
\end{remark}

\section{Construction of finite cochain complex (Dolbeault case)}
In this case we can say almost same argument for de Rham case without difficulties.
Let $G$ be a simply connected  solvable Lie group and $\g$ be the Lie algebra of $G$.
Let $N$ be the nilradical  of $G$.
Let $\sigma:G\to GL(V_{\sigma})$ be a holomorphic representation.
Suppose the restriction $\sigma_{\vert_{N}}$ is unipotent.
We consider the direct sum
\[\bigoplus_{\alpha\in {\mathcal B}_{(G,N)}}\bigwedge {\frak g}^{\ast}\otimes V_{\alpha}\otimes V_{\rho}.\]
Then we have the $G$-action on this cochain complex via $\bigoplus {\rm Ad}\otimes \alpha\otimes \rho$.
Consider  the semi-simple part 
\[((\bigoplus_{\alpha\in{\mathcal B}_{(G,N)}} {\rm Ad}\otimes \alpha\otimes \rho)(g))_{s}=\bigoplus_{\alpha\in{\mathcal B}_{(G,N)}} ({\rm Ad}_{g})_{s}\otimes \alpha(g)\otimes (\rho(g))_{s}.
\]
Take a  simply connected complex nilpotent subgroup $C\subset G$ as Proposition \ref{Cde} and Remark \ref{CCO}.
 Since $C$ is nilpotent, the map 
\[\Phi:C\ni c \mapsto \bigoplus_{\alpha\in{\mathcal B}_{(G,N)}} ({\rm Ad}_{g})_{s}\otimes \alpha(g)\otimes (\rho(g))_{s} \in {\rm Aut}(\bigoplus_{\alpha\in{\mathcal B}_{(G,N)}}\bigwedge {\frak g}^{\ast}\otimes V_{\alpha}\otimes V_{\sigma})\]
is a homomorphism.
We denote by 
\[(\bigoplus_{\alpha\in{\mathcal A}_{(G,N)}}\bigwedge {\frak g}^{\ast}\otimes V_{\alpha}\otimes V_{\sigma})^{\Phi(C)}
\]
the subcomplex consisting of the $\Phi(C)$-invariant elements.
By  similar proof of Lemma \ref{incis}, we have:
\begin{lemma}\label{iddd}
The inclusion
\[(\bigoplus_{\alpha\in{\mathcal B}_{(G,N)}}\bigwedge {\frak g}^{\ast}\otimes V_{\alpha}\otimes V_{\sigma})^{\Phi(C)}\subset \bigoplus_{\alpha\in{\mathcal B}_{(G,N)}}\bigwedge {\frak g}^{\ast}\otimes V_{\alpha}\otimes V_{\sigma}
\]
induces a cohomology isomorphism.
\end{lemma}

We have a basis $X_{1},\dots,X_{n}$ of $\g$ such that $({\rm Ad}_{c})_{s}={\rm diag} (\alpha_{1}(c),\dots,\alpha_{n}(c))$ for $c\in C$.
Let $x_{1},\dots, x_{n}$ be the basis of $\g^{\ast}$ which is dual to $X_{1},\dots ,X_{n}$.
We have a basis $v_{1},\dots ,v_{m}$ of $V_{\sigma}$ such that 
$(\sigma(c))_{s}={\rm diag}(\alpha_{1}^{\prime}(c),\dots ,\alpha^{\prime}_{m}(c))$ for any $c\in C$.
Let $v_{\alpha}$ be a basis of $V_{\alpha}$ for each character $\alpha \in{\mathcal B}_{(G,N)}$.
By $G=C\cdot N$, we have $G/N=C/C\cap N$ and hence we have ${\mathcal B}_{(G,N)}={\mathcal B}_{C,C\cap N}=\{\alpha\in {\rm Hom}_{hol}(C,\C^{\ast}) \vert \alpha_{\vert_{C\cap N}}=1\}$.

For a multi-index $I=\{i_{1},\dots ,i_{p}\}$ we write $x_{I}=x_{i_{1}}\wedge\dots \wedge x_{i_{p}}$,  and $\alpha_{I}=\alpha_{i_{1}}\cdots \alpha_{i_{p}}$.
We consider the basis 
\[\{ x_{I} \otimes v_{\alpha}\otimes v_{k}\}_{I\subset \{1,\dots, n\}, \alpha\in {\mathcal A}_{C,C\cap N}, k\in \{1,\dots, m \}}
\]
of $\bigoplus_{\alpha\in {\mathcal B}_{C,C\cap N}}\bigwedge {\frak g}^{\ast}\otimes V_{\alpha}\otimes V_{\sigma}$.
Since the action
\[\Phi:C\to{\rm Aut}(\bigoplus_{\alpha}\bigwedge {\frak g}^{\ast}\otimes V_{\alpha}\otimes V_{\sigma})\]
is  the semi-simple part of $(\bigoplus {\rm Ad}\otimes \alpha\otimes \sigma)_{\vert_C}$, we have
\[
\Phi(a)(x_{I} \otimes v_{\alpha}\otimes v_{k})
=\alpha^{-1}_{I}\alpha \alpha_{k}^{\prime} x_{I}\otimes v_{\alpha}\otimes v_{k}.
\]

Hence we have
\begin{multline*}
(\bigoplus_{\alpha}\bigwedge {\frak g}^{\ast}\otimes V_{\alpha}\otimes V_{\sigma})^{\Phi(C)}\\
=\langle x_{I}\otimes v_{\alpha_{I}\alpha^{\prime-1}_{k}} \otimes v_{k}\rangle_{I\subset \{1,\dots, n\},  k\in \{1,\dots, m \}}\\
=(\bigwedge \langle x_{1}\otimes v_{\alpha_{1}},\dots, x_{n}\otimes v_{\alpha_{n}} \rangle) \otimes \langle v_{\alpha^{\prime -1}_{1}}\otimes v_{1}, \dots ,v_{\alpha^{\prime -1}_{m}}\otimes v_{m}  \rangle.
\end{multline*}

\begin{corollary}\label{MMTT}
Let $B^{\ast}_{\Gamma}$ be the subcomplex of $\langle  x_{I}\otimes v_{ \alpha_{I}\alpha^{\prime-1}_{k}} \otimes v_{k}\rangle_{I\subset \{1,\dots, n\},  k\in \{1,\dots, m \}}$ defined as
\[B^{\ast}_{\Gamma}=\left\langle  x_{I}\otimes v_{\alpha_{I}\alpha^{\prime-1}_{k}} \otimes v_{k}\vert \left(\frac{\bar\alpha_{I}\bar\alpha^{\prime-1}_{k}}{\alpha_{I}\alpha^{\prime-1}_{k}}\right)_{ \vert_{\Gamma}}=1\right\rangle.
\]
Then we have an isomorphism 
\[H^{\ast}(B^{\ast}_{\Gamma})\cong H^{0,\ast}(G/\Gamma,  L_{\bar\sigma}).
\]
\end{corollary}
\begin{proof}
It is known that we have the $1$-$1$ correspondence between the isomorphism classes of flat holomorphic line bundles over a complex torus  and the unitary characters of its lattice (see \cite{Po}).
By this, for $\alpha \in {\mathcal B}_{(G,N)}$, considering the unitary character $\frac{\bar\alpha}{\alpha}$, the holomorphic line bundle $L_{\bar\alpha}$ is trivial if and only if the restriction $(\frac{\bar\alpha}{\alpha})_{\vert_{\Gamma}}$ is trivial.
Hence 
\[\iota(x_{I}\otimes v_{\alpha_{I}\alpha^{\prime-1}_{k}} \otimes v_{k})\in A^{\ast}(G/\Gamma,  L_{\bar\sigma})\]
 if and only if the restriction $(\frac{\bar\alpha_{I}\bar\alpha^{\prime-1}_{k}}{\alpha_{I}\alpha^{\prime-1}_{k}})_{\vert_{\Gamma}}$   is trivial. 
Then we have $\iota^{-1}(A^{\ast}(G/\Gamma,  L_{\bar\sigma}))=B^{\ast}_{\Gamma}$.
\end{proof}

\begin{corollary}\label{IMG}
We consider the following condition:

$(\star)$ For each multi-index $I=\{i_{1},\dots ,i_{p}\}$ and $k\in \{1,\dots ,m\}$, the character $\alpha_{I}\alpha^{\prime-1}_{k}$ is trivial if and only if the restriction $(\frac{\bar\alpha_{I}\bar\alpha^{\prime-1}_{k}}{\alpha_{I}\alpha^{\prime-1}_{k}})_{\vert_{\Gamma}}$ is trivial.

If the condition $(\star)$ holds, then we have an isomorphism
\[H^{\ast}(\g, V_{\sigma})\cong H^{0,\ast}(G/\Gamma, L_{\bar\sigma}).
\]

\end{corollary}
\begin{proof}
Suppose  the condition $(\star)$  holds.
Then we have $B^{\ast}_{\Gamma}= (\bigwedge \g^{\ast}\otimes V_{\sigma})^{\Phi(C)}$.
Hence we have
\[H^{0,\ast}(G/\Gamma,L_{\bar\sigma})\cong H ^{\ast}(\bigwedge \g^{\ast}\otimes V_{\sigma})^{\Phi(C)}\cong H^{\ast}(\g, V_{\sigma}).
\]

\end{proof}

\begin{remark}

We define the nilpotent Lie algebra ${\frak u}_{G}$ as Remark
\ref{shd}.
In complex case, ${\frak u}_{G}$ is also a complex Lie algebra.
As similar to Remark
\ref{shd}, we have
\[
(\bigoplus_{\alpha}\bigwedge {\frak g}^{\ast}\otimes V_{\alpha})^{\Phi(C)}
=\bigwedge \langle x_{1}\otimes v_{\alpha_{1}},\dots, x_{n}\otimes v_{\alpha_{n}} \rangle\cong \bigwedge{\frak u}_{G}^{\ast}.
\]

Suppose $G$ has a lattice $\Gamma$.
We consider the cochain complex
\[B^{\ast}_{\Gamma}=\left\langle  x_{I}\otimes v_{\alpha_{I}} \vert \left(\frac{\bar\alpha_{I}}{\alpha_{I}}\right)_{ \vert_{\Gamma}}=1\right\rangle.
\]
Then we have an isomorphism $H^{0,\ast}(B^{\ast}_{\Gamma})\cong H^{0,\ast}(G/\Gamma)$ by Corollary \ref{MMTT}.
We consider the following condition.

$(\Box)$ For each $1\le i\le n$, the restriction $(\frac{\bar\alpha_{i}}{\alpha_{i}})_{ \vert_{\Gamma}}$  is trivial.\\
If the condition $(\Box)$ holds, then we have
\[B^{\ast}_{\Gamma}=\bigwedge \langle x_{1}\otimes v_{\alpha_{1}},\dots, x_{n}\otimes v_{\alpha_{n}} \rangle\cong \bigwedge{\frak u}_{G}^{\ast}.\]

Let $U_{G}$ be the simply connected complex Lie group with the Lie algebra ${\frak u}_{G}$.
Then $U_{G}$ is the nilradical of the semi-simple splitting of $G$ (see \cite{dek}).
It is known that if $G$ has a lattice, then   $U_{G}$  has a lattice $\Gamma^{\prime}$ (see \cite{Aus}).

Hence we have:
\begin{corollary}\label{NIM}
Let $G$ be a simply connected complex solvable Lie group with a lattice $\Gamma$.
If the condition $(\Box)$ holds, then there exists a complex parallelizable nilmanifold $U_{G}/\Gamma^{\prime}$ such that we have an isomorphism
\[H^{\ast,\ast}(G/\Gamma)\cong H^{\ast,\ast}(U_{G}/\Gamma^{\prime}).
\]
\end{corollary}
By this corollary we have some solvmanifolds whose Dolbeault cohomology is isomorphic to the Dolbeault cohomology of nilmanifolds.
\end{remark}

\section{example}
Let $G=\C\ltimes_{\phi} \C^{2}$ such that \[\phi(z)=\left(
\begin{array}{cc}
e^{z}& 0  \\
0&    e^{-z}  
\end{array}
\right).\]
Then we have $a+\sqrt{-1}b, c+\sqrt{-1}d\in \C$ such that $ \Z(a+\sqrt{-1}b)+\Z(c+\sqrt{-1}d)$ is a lattice in $\C$ and
$\phi(a+\sqrt{-1}b)$ and $\phi(c+\sqrt{-1}d)$
 are conjugate to elements of $SL(4,\Z)$ where we regard  $SL(2,\C)\subset SL(4,\R)$ (see \cite{Hd}).
Hence we have a lattice $\Gamma=(\Z(a+\sqrt{-1}b)+\Z( c+\sqrt{-1}d))\ltimes_{\phi} \Gamma^{\prime\prime}$ such that $\Gamma^{\prime\prime}$ is a lattice of $\C^{2}$.

\subsection{Twisted de Rham cohomology $H^{1}(G/\Gamma, E_{\rm Ad})$}
For a coordinate $(w,z_{1},z_{2})\in \C\ltimes_{\phi} \C^{2}$ we have the basis $\{v_{1},\dots ,v_{6}\}$ of $\g_{\C}$ such that
\[v_{1}=e^{w}\frac{\partial}{\partial z_{1}}, v_{2}=e^{\bar w}\frac{\partial}{\partial \bar z_{1}},v_{3}=e^{- w}\frac{\partial}{\partial  z_{2}}, v_{4}=e^{-\bar w}\frac{\partial}{\partial \bar z_{2}}, v_{5}=\frac{\partial}{\partial  w}, v_{6}=\frac{\partial}{\partial \bar w}.
\]
Consider the dual basis
\[e^{-w}dz_{1},e^{-\bar w} d\bar z_{1}, e^{w}dz_{2}, e^{\bar w} d\bar z_{2}, dw, d\bar w.
\]
As we consider $\g_{\C}$ as a representation of $\g$ via $\rm Ad$, we have the cochian complex 
$\bigwedge \g^{\ast}\otimes \g_{\C}$ whose differential is given by
\[dv_{1}=dw\otimes v_{1},\, dv_{2}=d\bar w \otimes v_{2},\, dv_{3}=-dw\otimes v_{3}, \,dv_{4}=-d\bar w\otimes v_{4}
\]
\[dv_{5}=-e^{-w}dz_{1}\otimes v_{1}+e^{w}dz_{2}\otimes v_{3},\, dv_{6}=-e^{\bar w}d\bar z_{1}\otimes v_{2}+e^{\bar w}d\bar z_{2}\otimes v_{4}.\]
For $(w,0,0)\in \C$, we have $({\rm Ad}_{(w,0,0)})_{s}={\rm diag} (e^{w},e^{\bar w}, e^{-w},e^{-\bar w},1,1)$ for the basis $\{v_{1},\dots ,v_{6}\}$.
Consider the cochain complex 
\[(\bigoplus_{\alpha}\bigwedge {\frak g}^{\ast}\otimes V_{\alpha}\otimes V_{\rho})^{\Phi(C)} 
\]
as Section \ref{cos} where $C=\C$.
Then we have
\begin{multline*}
(\bigoplus_{\alpha}\bigwedge {\frak g}^{\ast}\otimes V_{\alpha}\otimes V_{\rho})^{\Phi(C)}\\
=\bigwedge \langle ^{-w}dz_{1}\otimes v_{e^{w}} ,\, e^{-\bar w} d\bar z_{1}\otimes v_{e^{\bar w}}, \, e^{w}dz_{2}\otimes v_{e^{- w}},\, e^{\bar w} d\bar z_{2}\otimes v_{e^{-\bar w}},\, dw,\, d\bar w \rangle\\
\bigotimes  \langle v_{1}\otimes v_{e^{- w}},\, v_{2}\otimes v_{e^{-\bar w}},v_{3} \otimes v_{e^{w}},\, v_{4}\otimes v_{e^{\bar w}} ,\, v_{5},\, v_{6}\rangle.
\end{multline*}

For any lattice $\Gamma$ we have $b_{1}(G/\Gamma)=b_{1}(\g)=2$.
But we will see that  $\dim H^{1}(G/\Gamma, e_{\rm Ad})$ varies for a choice of $\Gamma$.
If $b,d \in \pi\Z$, then we have 
\[A^{0}_{\Gamma}=\langle v_{5}, v_{6}\rangle,
\]
\begin{multline*}
A^{1}_{\Gamma}=\langle e^{-w}dz_{1}\otimes v_{1},\, e^{-w}dz_{1}\otimes v_{e^{w}}\otimes v_{2}\otimes v_{e^{-\bar w}},\\
 e^{-\bar w}d\bar z_{1}\otimes v_{e^{\bar w}}\otimes v_{1}\otimes v_{e^{ w}},\,  e^{-\bar w}d\bar z_{1}\otimes v_{2},\\
e^{w}dz_{2}\otimes v_{3},\,  e^{w}dz_{2}\otimes v_{e^{-w}}\otimes v_{4}\otimes v_{e^{\bar w}},\\
e^{\bar w}d\bar z_{2}\otimes v_{e^{-\bar w}}\otimes v_{3}\otimes v_{e^{ -w}},\,  e^{\bar w}d\bar z_{2}\otimes v_{4},\\
dw\otimes v_{5},\, dw\otimes v_{6},\, d\bar w\otimes v_{5},\,  d\bar w\otimes v_{6}\rangle.
\end{multline*}
Hence we have $\dim H^{1}(G/\Gamma, V_{\rm Ad})=\dim H^{1}(A_{\Gamma}^{\ast})=6$.

On the other hand, if $b\not \in \pi\Z$ or $d\not \in\pi\Z$, then we have
\[A^{0}_{\Gamma}=\langle v_{5}, v_{6}\rangle,
\]
\begin{multline*}
A^{1}_{\Gamma}=\langle e^{-w}dz_{1}\otimes v_{1},\,
  e^{-\bar w}d\bar z_{1}\otimes v_{2},\,
e^{w}dz_{2}\otimes v_{3}, \,
 e^{\bar w}d\bar z_{2}\otimes v_{4},\\
dw\otimes v_{5},\, dw\otimes v_{6},\, d\bar w\otimes v_{5},\, d\bar w\otimes v_{6}\rangle.
\end{multline*}
Hence we have $\dim H^{1}(G/\Gamma, E_{\rm Ad})=\dim H^{1}(A_{\Gamma}^{\ast})=2$.

\subsection{Dolbeault cohomology $H^{\ast,\ast}_{\bar\partial}(G/\Gamma)$}
For a coordinate $(z_{1},z_{2},z_{3})\in \C\ltimes_{\phi} \C^{2}$, we consider the basis $(x_{1},x_{2},x_{3})=(d\bar z_{1},e^{-\bar z_{1}}d\bar z_{2},e^{\bar z_{1}}d\bar z_{3})$ of $\g^{\ast}$.
We consider  $C=\C=\{(z_{1})\}$ and $(\alpha_{1},\alpha_{2},\alpha_{3})=(1,e^{z_1},e^{-z_{1}})$ for $C$ and $(\alpha_{1},\alpha_{2},\alpha_{3})$ as in Section \ref{cos}.
If $b\not \in \pi\Z$ or $c\not \in\pi\Z$, then $(\star)$ holds and hence we have $H^{\ast,\ast}_{\bar\partial}(G/\Gamma)\cong \bigwedge\C^{3}\otimes H^{\ast}(\g)$.
If $b,d \in \pi\Z$, then the condition $(\Box)$ holds and hence we have $H^{\ast,\ast}_{\bar\partial}(G/\Gamma)\cong \bigwedge\C^{3}\otimes  \bigwedge\C^{3}$.
There exists a lattice $\Gamma$ which satisifies the condition $(\star)$ or $(\Box)$ (see \cite{Hd}).

{\bf  Acknowledgements.} 

The author would like to express his gratitude to   Toshitake Kohno for helpful suggestions and stimulating discussions.
This research is supported by JSPS Research Fellowships for Young Scientists.


\begin{thebibliography}{40}

\bibitem{Aus}
L. Auslander, An exposition of the structure of solvmanifolds. I. Algebraic theory.  Bull. Amer. Math. Soc. {\bf 79}  (1973), no. 2, 227--261.
\bibitem{dek} K.  Dekimpe,
Semi-simple splittings for solvable Lie groups and polynomial structures. Forum Math. {\bf 12} (2000), no. 1, 77--96.
\bibitem{Dek} K. Dekimpe,  Solvable Lie algebras, Lie groups and polynomial structures,  Compositio Math. {\bf 121} (2000), no. 2,  183--204.
\bibitem{DER} N. Dungey, A. F. M. ter Elst, D. W. Robinson,  Analysis on Lie Groups with Polynomial Growth, Birkh\"auser (2003).
\bibitem{Hd} K. Hasegawa, Small deformations and non-left-invariant complex structures on six-dimensional compact solvmanifolds. Differential Geom. Appl. {\bf 28} (2010), no. 2, 220--227.
\bibitem{Hatt} A. Hattori, Spectral sequence in the de Rham cohomology of fibre bundles. J. Fac. Sci. Univ. Tokyo Sect. I {\bf 8} 1960 289--331 (1960). 
\bibitem{Hir} F. Hirzebruch, Topological Methods in Algebraic Geometry, third enlarged ed.,  Springer-Verlag, 1966.
\bibitem{Hd} K. Hasegawa, Small deformations and non-left-invariant complex structures on six-dimensional compact solvmanifolds. Differential Geom. Appl. {\bf 28} (2010), no. 2, 220--227.
\bibitem{hum} J. E. Humphreys, Linear algebraic groups. Springer-Verlag, New York 1981
\bibitem{K2} H. Kasuya, Minimal models, formality and hard Lefschetz properties of solvmanifolds with local systems. To appear in J. Differential Geometry.
http://arxiv.org/abs/1009.1940.
\bibitem{Kd}
H. Kasuya, Techniques of computations of Dolbeault cohomology of  solvmanifolds.    Math. Z DOI: 10.1007/s00209-012-1013-0 (online first).  arXiv:1107.4761 (preprint)

\bibitem{Mosc} G. D. Mostow,  Cohomology of topological groups and solvmanifolds. Ann. of Math. (2) {\bf 73} 1961 20--48.
\bibitem{Nom}
K. Nomizu, On the cohomology of compact homogeneous spaces of nilpotent Lie groups. 
Ann. of Math. (2) {\bf 59}, (1954). 531--538.

\bibitem{OV}
A. L. Onishchik, E. B. Vinberg (Eds), Lie groups and Lie algebras I\hspace{-.1em}I, Springer (2000).
\bibitem{Po} A. Polishchuk, Abelian Varieties, Theta Functions and the Fourier Transform. Cambridge University Press 2002.
 \bibitem{R}
 M.S. Raghnathan, Discrete subgroups of Lie Groups, Springer-Verlag, New York, 1972.
\bibitem{Sak} Y. Sakane, On compact complex parallelisable solvmanifolds. Osaka J. Math. {\bf 13} (1976), no. 1, 187--212.
\bibitem{St} N. Steenrod, The Topology of Fibre Bundles, Princeton University Press
(1951).
\end{thebibliography}
\end{document}